\documentclass[10pt,a4paper]{amsart}
\usepackage[utf8]{inputenc}
\usepackage{amsfonts,amsthm,amsmath,amssymb,mathtools,bbm}
\usepackage[all,arc]{xy}
\usepackage{enumerate,enumitem}
\usepackage{mathrsfs}
\usepackage{graphicx}
\usepackage[margin=2cm]{geometry}
\usepackage{tikz-cd}
\usepackage{etoolbox}
\usepackage{csquotes}
\usepackage{hyperref}
\usepackage{float}
\usepackage{placeins}
\usepackage{mathabx,epsfig}
\def\acts{\mathrel{\reflectbox{$\righttoleftarrow$}}}

\usepackage{silence}
\DeclareMathOperator{\Rep}{Rep}
\DeclareMathOperator{\NC}{NC}

\newcommand{\stbinom}{\genfrac{\{}{\}}{0pt}{}}
\newcommand{\qbinom}{\genfrac{[}{]}{0pt}{}}
\def\multiset#1#2{\ensuremath{\left(\kern-.3em\left(\genfrac{}{}{0pt}{}{#1}{#2}\right)\kern-.3em\right)}}

\theoremstyle{plain}
\newtheorem{thm}{Theorem}[section]

\newtheorem{cor}[thm]{Corollary}
\newtheorem{prop}[thm]{Proposition}
\newtheorem{lem}[thm]{Lemma}
\newtheorem{conj}[thm]{Conjecture}

\newtheorem{fact}[thm]{Fact}

\theoremstyle{definition}
\newtheorem{defn}[thm]{Definition}

\newtheorem{exmp}[thm]{Example}

\newtheorem{notn}[thm]{Notation}

\let\olddefn\defn
\renewcommand{\defn}{\olddefn\normalfont}

\newtheorem{rem}[thm]{Remark}

\makeatletter
\let\c@equation\c@thm
\makeatother
\numberwithin{equation}{section}


\newcommand{\cmt}[1]{\noindent {\bf \color{red} #1 \color{black}}}

\title{Dihedral Sieving Phenomena}
\author{Sujit Rao}
\address{Department of Mathematics, Cornell University, Ithaca, NY 14850}
\email{sujitkrao@gmail.com}
\author{Joe Suk}
\address{Department of Mathematics, Stony Brook University, Stony Brook, NY 11794-
3651, USA}
\email{ybjosuk@gmail.com}

\begin{document}
\begin{abstract}
Cyclic sieving is a well-known phenomenon where certain interesting polynomials, especially $q$-analogues, have useful interpretations related to actions and representations of the cyclic group.
We propose a definition of sieving for an arbitrary group $G$ and study it for the dihedral group $I_2(n)$ of order $2n$. This requires understanding the generators of the representation ring of the dihedral group.
For $n$ odd, we exhibit several instances of dihedral sieving which involve the generalized Fibonomial coefficients, recently studied by Amdeberhan, Chen, Moll, and Sagan. We also exhibit an instance of dihedral sieving involving Garsia and Haiman's $(q,t)$-Catalan numbers.
\end{abstract}
\maketitle

\section{Introduction}
The cyclic sieving phenomenon was originally studied by Reiner, Stanton, and White in \cite{rsw} in 2004 and has, since then, led to a greater understanding of the combinatorics of various finite sets with a natural cyclic action.
In particular, cyclic sieving allows one to count the fixed points of a cyclic action on a finite set through an associated generating function. These generating functions often appear in other contexts, such as the generating function for permutation statistics related to Coxeter groups and as the Hilbert series of interesting graded rings. Proofs of cyclic sieving also tend to have interesting connections with representation theory.

We start by precisely defining cyclic sieving:

\begin{defn}[cyclic sieving phenomenon]\label{defn-cyclicsievng}
Let $X$ be a finite set, $X(q)$ be a polynomial with nonnegative integral coefficients, and $\mathbb{Z}/n\mathbb{Z}$ be a cyclic group of order $n$ with a group action on $X$. Let $\omega_n:\mathbb{Z}/n\mathbb{Z}\hookrightarrow \mathbb{C}^{\times}$ be the map defined by $m\mapsto e^{2\pi m i/n}$. Then, we say the triple $(X,X(q),\mathbb{Z}/n\mathbb{Z})$ exhibits the \emph{cyclic sieving phenomenon} if for all $c\in \mathbb{Z}/n\mathbb{Z}$,
\begin{equation}
|X(q)|_{q=\omega(c)}=|\{x\in X:c(x)=x\}|
\end{equation}
\end{defn}
As $X(1) = |X|$, $X(q)$ can be considered a $q$-analogue of the cardinality. Before discussing some classic examples of cyclic sieving, we recall the definition of the $q$-binomial coefficient $\qbinom{n}{k}_q$. First, let $[n]_q=1+q+q^2+\cdots+q^{n-1}$ and let $[n]!_q=[n]_q[n-1]_q\cdots[2]_q[1]_q$. Then, the $q$-binomial coefficient is defined as
\begin{equation}
\qbinom{n}{k}_q:=\frac{[n]!_q}{[k]!_q[n-k]!_q}
\end{equation}
This is a rational function in $q$. It is not immediately obvious, but the $q$-binomial coefficient can also be shown to be a polynomial in $q$ with nonnegative integral coefficients. MacMahon's $q$-Catalan number, defined similarly,
\begin{equation}
C_n(q) \coloneqq \frac{1}{[n+1]_q}\qbinom{2n}{n}_q
\end{equation}
is also a polynomial in $q$ with nonnegative integral coefficients. For one proof, see Theorem 1.6 of \cite{haglundq}.

Now, we discuss some examples of cyclic sieving given in the seminal paper \cite{rsw}. All of these are cases where the natural cyclic action of $C$ on the set $[n]:=\{1,\ldots,n\}$ induces an action on some collection of subsets of $[n]$. One of these collections is the collection of non-crossing partitions: a partition of $[n]$ is non-crossing if its blocks do not cross when drawn on a disk whose boundary is labeled clockwise with $1,2,\ldots,n$. Sometimes the action $C\acts [n]$ is interpreted geometrically via rotations of a regular $n$-gon.
\begin{exmp}\label{ex:sieving-examples}
Let $C$ be a cyclic group of order $n$. Then, the following triples $(X,X(q),C)$ exhibit cyclic sieving phenomenon.
\begin{enumerate}
\item Let $X=\{\text{size $k$ multisubsets of $[n]$}\}$ and let $X(q)=\qbinom{n+k-1}{k}_q$.
\item Let $X=\{\text{size $k$ subsets of $[n]$}\}$ and let $X(q)=\qbinom{n}{k}_q$.
\item Let $X=\{\text{noncrossing partitions of an $n$-gon}\}$ and let $X(q)=C_n(q)$.
\item Let $X=\{\text{noncrossing partitions of an $n$-gon using $n-k$ parts}\}$ and let $X(q)=\frac{1}{[n]_q}\qbinom{n}{k}_q\qbinom{n}{k+1}_q q^{k(k+1)}$.
\item Let $X=\{\text{triangulations of a regular $n$-gon}\}$ and let $X(q)=C_{n-2}(q)$.
\item Let $X=\{\text{dissections of a convex $n$-gon using $k$ diagonals}\}$ and let $X(q)=\frac{1}{[n+k]_q}\qbinom{n+k}{k+1}_q\qbinom{n-3}{k}_q$.
\end{enumerate}

\end{exmp}
Note that in case (6), $X(q)$ is a $q$-analogue of $\frac{1}{n+k}\binom{n+k}{k+1}\binom{n-3}{k}$, a formula for the number of dissections of a convex $n$-gon using $k$ diagonals and in case (4), $X(q)$ is a $q$-analogue of $N(n,k)=\frac{1}{n}\binom{n}{k}\binom{n}{k+1}$, the \emph{Narayana number} which counts the number of non-crossing partitions of $[n]$ using $n-k$ parts. Both of these are also polynomials in $q$ for virtually the same reasons their corresponding formulas $X(1)$ are integers. Also note that case (6) is a specific example of case (5) where $k=n-3$.

Every cyclic action can be equipped with a generating polynomial $X(q)$ to obtain an instance of cyclic sieving. If we require $X(q)$ be of degree at most $n-1$, then the choice of polynomial is unique. Thus, cyclic sieving can be made ubiquitous. However, the interest and fascination of this subject stem from the fact that some of these associated generating polynomials will arise in other contexts. In particular, those arising via $q$-binomials or $q$-Catalan numbers are of particular interest and their importance is discussed in \cite{sagan}.

Typically, we have an educated guess for an appropriate cyclic sieving polynomial for a particular cyclic action which has already been used in another context. Proving that a chosen polynomial produces an instance of cyclic sieving often happens in two ways: via direct computation of the generating polynomial or an understanding of the corresponding permutation representation. The following theorem is used in the latter way.

Defining sieving phenomena for other finite groups is a natural generalization for understanding other group actions. For the case of abelian groups, or direct products of cyclic groups, a natural definition of polycyclic sieving using a multivariate generating function was defined with accompanying instances, or examples, in \cite{brs}.

In this work, we give a new definition of sieving phenomena for finite groups motivated by the representation theoretic perspective of cyclic sieving provided above and apply this to the case of dihedral group actions. We prove that the natural dihedral action in situations (1)--(5) in Example \ref{ex:sieving-examples} have dihedral sieving for $n$ odd. The analogous generating polynomials we obtain for situations (1)--(4) are defined in terms of generalized Fibonacci polynomials, first studied by Hoggatt and Long in \cite{hl}, and their induced generalized Fibonomial coefficients, recently studied by Amdeberhan, Chen, Moll, and Sagan in \cite{acms}. We find the generating polynomial for dihedral sieving on the set of triangulations of a regular $n$-gon is related to Garsia and Haiman's $(q,t)$-Catalan polynomial.

\section{Preliminaries}
Throughout the paper, we will use the following notation.
\begin{notn}
Let $C=\langle c\rangle$ be a cyclic group of order $n$. Let $\omega : C \to \mathbb{C}^{\times}$ be an embedding of $C$ defined by $c\mapsto e^{2\pi i/n}$. This can also be considered a $1$-dimensional complex representation of $C$.

If $V$ is a representation of a group $G$, we will use $\chi_V$ to refer to its character. If $x\in G$, $\chi_V(x)$ is the value of the character at $x$. If $C\subset G$ is a conjugacy class, then $\chi_V(C)$ is the value of the character on $C$ as a class function. In the case where $G=GL_N(\mathbb{C})$, we use $\chi_V(\operatorname{diag}(x_1,\ldots,x_N))$ to denote the value of the character on any diagonalizable element of $GL_N(\mathbb{C})$ having eigenvalues $x_1,\ldots,x_N$.
\end{notn}

We will first give an equivalent definition of cyclic sieving based on representation theory, which is more suitable for adapting to other groups.

\begin{defn}
Let $G$ be a group and let $A$ be the set of isomorphism classes of finite-dimensional $G$-representations over $\mathbb{C}$. The \emph{representation ring} of $G$ with coefficients in $\mathbb{Z}$ is
\[\Rep(G) = \mathbb{Z}[A]/(I + J)\]
where $\mathbb{Z}[A]$ is the polynomial ring over $\mathbb{Z}$ freely generated by $A$, and $I$ and $J$ are the ideals
\begin{align*}
I &= (\{[U \oplus V] - ([U] + [V])\}) \\
J &= (\{[U \otimes V] - [U][V]\})
\end{align*}
and $[U]$ denotes the isomorphism class of a $G$-representation $U$.
\end{defn}

We will often use the following facts about the representation ring of a group.

\begin{fact}
The representation ring $\Rep(G)$ is a free abelian group with a basis given by isomorphism classes of irreducible representations.
\end{fact}

\begin{fact}
The map defined by
\[\begin{tikzcd}[row sep={0pt},/tikz/column 1/.append style={anchor=base east},/tikz/column 2/.append style={anchor=base west}]
\Rep(G) \arrow[r] & {\displaystyle \bigoplus _{\substack{\text{all conj.}\\\text{classes in $G$}}} \mathbb{C}} \\
{[V]} \arrow[r, mapsto] & \chi_{V}
\end{tikzcd}\]
which sends an isomorphism class of a representation to its character (in the ring of conjugacy class functions on $G$ with pointwise multiplication), is an injective ring homomorphism whose image is the $\mathbb{Z}$-span of characters of irreducible representations.
\end{fact}

The following theorem gives one definition of cyclic sieving based on the representation-theoretic perspective.

\begin{thm}[{\cite[Prop. 2.1]{rsw}}]
Consider a triple $(X,X(q),C)$ as in the setup of Proposition \ref{defn-cyclicsievng}. Let $A_X$ be a graded $\mathbb{C}$-vector space $A_X=\oplus_{i\geq 0}A_{X,i}$ having $\sum_{i\geq 0} \operatorname{dim}_{\mathbb{C}} A_{X,i} q^i=X(q)$. $A_X$ can be considered a representation of $C$ in which each $c\in C$ acts on the graded component $A_{X,i}$ by the scalar $\omega(c)^i$. Then, $(X,X(q),C)$ has cyclic sieving if and only if we have an isomorphism of $C$-representations $A_X\cong \mathbb{C}[X]$.
\end{thm}

We can now state an equivalent definition of cyclic sieving based on the representation ring.

\begin{thm}[{\cite[Prop. 2.1]{rsw}}]
Let $(X, X(q), C)$ be a triple where $C$ acts on $X$ and $X(q) \in \mathbb{N}[q]$. This triple has cyclic sieving if and only if $\mathbb{C}[{X}] = X(\omega)$ in $\Rep(C; \mathbb{Z})$.
\end{thm}

This motivates the following general definition, which is key to all results in this work.

\begin{defn}
Let $G$ be a group and $\rho_{1}, \dots, \rho_{k}$ be representations of $G$ over $\mathbb{C}$ which generate $\Rep(G)$ as a ring. Let $X$ be a $G$-set and $X(q_{1}, \dots, q_{k}) \in \mathbb{Z}[q_{1}, \dots, q_{k}]$. Then the quadruple $(X, X(q_{1}, \dots, q_{k}), (\rho_{1}, \dots, \rho_{k}), G)$ has \textbf{$G$-sieving} if $\mathbb{C}[{X}] = X(\rho_{1}, \dots, \rho_{k})$ in $\Rep(G)$.
\end{defn}

\begin{exmp}
Let $G = C$ and $\rho_{1} = \omega$. Then the definition $G$-sieving above agrees with the usual definition of cyclic sieving.
\end{exmp}

\begin{exmp}
Let $G = C_{n} \times C_{m}$ be a product of cyclic groups. Let $\rho_{1} = \omega_{n} \otimes 1_{m}$, where $\omega_{n} : C_{n} \to \mathbb{C}^{\times}$ is an embedding and $1_{m}$ is the trivial representation of $C_{m}$, and similarly let $\rho_{2} = 1_{n} \otimes \omega_{m}$. Then the definition of $G$-sieving above agrees with the definition of bicyclic sieving given in \cite{brs}.
\end{exmp}

\begin{rem}
Given any $G$-set $X$ and generators $\rho_{1}, \dots, \rho_{k}$ of $\Rep(G)$, there is always at least one polynomial $X(q_{1}, \dots, q_{k})$ which exhibits $G$-sieving for $X$. This follows directly from the fact that $\rho_{1}, \dots, \rho_{k}$ generate $\Rep(G)$. However, there is no guarantee that there is a canonical or interesting choice of such a polynomial, especially if there are complicated relations between the generators.
\end{rem}

\begin{rem}
Suppose we are given a set of points $\{a_{C}\} \in \mathbb{C}^{k}$ indexed by conjugacy classes in a group $G$ where $k \in \mathbb{N}$. Let $X$ be a finite $G$-set, and $p \in \mathbb{C}[q_{1}, \dots, q_{k}]$ such that $p(a_{C}) = \chi_{\mathbb{C}[X]}(C)$ for all conjugacy classes $C$. Then $(X, p(q_{1}, \dots, q_{k}), (\rho_{1}, \dots, \rho_{k}), G)$ exhibits $G$-sieving if $\rho_{i} \in \Rep(G)$ corresponds to the class function on $G$ defined by $C \mapsto (a_{C})_{i}$. If instead of $\Rep(G)$ we take $\mathbb{C} \otimes \Rep(G)$, then the virtual representations $\rho_{i}$ always exist.
\end{rem}

We are typically interested in cases where the polynomial $X(\cdot)$ can be written in an interesting way, such as product formulas based on $q$-analogues.

\section{Dihedral Sieving}\label{sec:dihedralsieving}
By the observations in the previous section, we can describe $I_2(n)$-sieving in terms of a generating set of $\Rep(I_2(n))$. We first need a description of the representation ring of the dihedral group. We start by briefly recalling the irreducible representations of $I_2(n)$. We adhere to the presentation
\begin{equation}\label{eq:dihedral}
I_2(n)=\langle r,s|r^n=s^2=e,rs=sr^{-1}\rangle
\end{equation}
The irreducible representations and the representation ring will depend on whether $n$ is odd or even. For $n$ odd, there are two $1$-dimensional irreducible representations, the trivial representation $\mathbbm{1}$ and the determinant representation, and $\lfloor n/2\rfloor$ $2$-dimensional irreducible representations: the representations $z_m$ which sends $r$ to a counterclockwise rotation matrix of $2\pi m/n$ radians and $s$ to a reflection matrix where $m\in[1,n/2)\cap\mathbb{N}$. For $n$ even, there are four $1$-dimensional irreducible representations: $\mathbbm{1}$, $\det$, $\chi_b$ which sends $\langle r^2,s\rangle$ to $1$ and $r$ to $-1$, and $\det\cdot \chi_b$. There are $(n/2-1)$ $2$-dimensional irreducible representations, defined the same way as the $n$ odd case. Character values for all of our group actions of interest are included in Table \ref{table-character}.

We refer to \cite{gaetz} for the following. First, we have (for both $n$ odd and even) the following relations among the irreducible representations:
\begin{align*}
{\det}^2 &=1\\
\det\cdot z_k&=z_k\\
z_{k+1}&=z_kz_1-z_{k-1}\text{ if $k\leq \frac{n-3}{2}$ and where $z_0=1+\det$}
\end{align*}
Thus, $\Rep(I_2(n))$ is generated by $\det,z_1$ for $n$ odd and by $\det,z_1,\chi_b$ for $n$ even.

\begin{rem}
Using the equation $z_{k + 1} = z_{k}z_{1} - z_{k-1}$ to define $z_{k}$ for all $k \in \mathbb{Z}$, it can be checked that $z_{0} = z_{n} = 1 + \det$. Thus $z_{1}$ by itself generates $\Rep(I_{2}(n))$ when $n$ is odd, and $z_{1}, \chi_{b}$ generate it when $n$ is even. However, the expression of $\det$ in terms of $z_{1}$ depends on $n$, so it is more useful to think of $I_{2}(n)$ as being generated by $z_{1}$ and $\det$.
\end{rem}

The examples of dihedral sieving we exhibit are all for odd $n$ and make use of the generalized Fibonacci polynomials and Fibonomial coefficients defined in \cite{acms}. We state some of their results here.

\begin{defn}[{\cite[Equation (2)]{acms}}]
The \emph{generalized Fibonacci polynomials} are a sequence $\{n\}_{s,t}$ of polynomials in $\mathbb{N}[s, t]$ defined inductively by
\begin{align*}
\{0\}_{s,t} &= 0 \\
\{1\}_{s,t} &= 1 \\
\{n + 2\}_{s,t} &= s\{n + 1\}_{s,t} + t\{n\}_{s,t}.
\end{align*}
We also define
\[\{n\}!_{s,t} = \{n\}_{s,t}\{n - 1\}_{s,t} \cdots \{1\}_{s,t}\]
and the \emph{Fibonomial coefficient} to be
\[\stbinom{n}{k}_{s,t} = \frac{\{n\}!_{s,t}}{\{k\}!_{s,t}\{n - k\}!_{s,t}}.\]
\end{defn}

\begin{prop}[{\cite[Theorem 5.2]{acms}}]
The Fibonomial coefficient $\stbinom{n}{k}_{s,t}$ is a polynomial in $s$ and $t$ with nonnegative integer coefficients.
\end{prop}

\begin{prop}[{\cite[Equation (7)]{acms}}]\label{prop:fibonomial-thm}
Let $X = \frac{s + \sqrt{s^{2} + 4t}}{2}$ and $Y = \frac{s - \sqrt{s^{2} + 4t}}{2}$. Then
\[\{n\}_{s,t} = Y^{n-1}\left. [n]_{q} \right|_{q = X/Y}\] and
\[\stbinom{n}{k}_{s,t} = Y^{k(n - k)} \left. \qbinom{n}{k}_{q} \right|_{q = X/Y}.\]
\end{prop}

The generalized Fibonacci polynomials have combinatorial interpretations related to tilings of rows of squares with monominoes and dominoes, which can be found in Section 1 of \cite{acms}. The Fibonomial coefficients have a similar interpretation related to tilings of a $k\times (n-k)$ rectangle containing a partition, which can be found in \cite{Sagan09combinatorialinterpretations}.

In this work, most of our dihedral sieving polynomials will be given in terms of generalized Fibonacci polynomials and Fibonomial coefficients. In all further sections, we will use the generators $z_{1}, -\det$ of $\Rep(I_{2}(n)$ and say that a triple $(X,P(s,t),I_2(n))$ has dihedral sieving if the quadruple $(X,P(s,t),(z_1,-\det),I_2(n))$ does.

\section{\texorpdfstring{Dihedral action on $k$-subsets and $k$-multisubsets of $\{1,\ldots,n\}$}{Dihedral action on k-subsets and k-multisubsets of \{1, ... n\}}}
We first recall some facts about cyclic sieving. For a rational representation $\rho:{GL}_n(\mathbb{C})\to{GL}_n(V)$, let $\chi_{\rho}(x_1,\ldots,x_N)$ be the trace on $V$ of any diagonalizable element of ${GL}_N(\mathbb{C})$ having eigenvalues $x_1,\ldots,x_N$.

\begin{thm}[{\cite[Lemma 2.4]{rsw}}]\label{thm:rsw-lemma-24}
Let $\rho:{GL}_n(\mathbb{C})\to{GL}_n(V)$ be a representation. Assume $V$ has a basis $\{v_x\}_{x\in X}$ which is permuted by $\mathbb{Z}/n\mathbb{Z}$ in the following way:
$$c(v_x)=v_{c(x)}\text{ for all $c\in \mathbb{Z}/n\mathbb{Z}, x\in X$}$$
Then, let $X(q)$ be the principal specialization 
$$X(q)=\chi_{\rho}(1,q,\ldots,q^{N-1})$$
Then, $(X,X(q),C)$ exhibits the cyclic sieving phenomenon.
\end{thm}

The above lemma can be used to prove cyclic sieving for $\mathbb{Z}/n\mathbb{Z}\acts X$ where $X=\binom{[n]}{k}$. Specifically, we take $V=V^{\lambda}$, the irreducible representation of $GL_n(\mathbb{C})$ with highest weight $\lambda=(k)\vdash k$ (the partition of $k$ with one part), and the specialization of the character value becomes the $q$-analogue of the Weyl character formula or the hook-content formula
\begin{equation}\label{eq-qhook}
\chi_{\rho}(1,q,\ldots,q^{n-1})=s_{\lambda}(1,q,\ldots,q^{n-1})=q^{b(\lambda)}\prod_{\text{cells $x\in\lambda$}}\frac{[n+c(x)]_q}{[h(x)]_q}=\qbinom{n+k-1}{k}_q
\end{equation}
where $h(x)$ is the hook-length of $\lambda$ at $x$ (total number of cells weakly to the right of $x$ or strictly below $x$), $c(x)$ is the hook content $j - i$ when cell $x$ is in row $i$ and column $j$, and
\[
b(\lambda) = \sum _{i} (i - 1)\lambda_{i}.
\]
We prove a generalization of this theorem which will help prove dihedral sieving for $k$-subsets and $k$-multisubsets. Given a highest-weight $\operatorname{GL}_n(\mathbb{C})$-representation $V$, consider its character specialization
\[ \chi_{V}(\operatorname{diag}(a^{n - 1}, a^{n - 2}b, \dots, ab^{n - 2}, b^{n - 1}))\]
as in Theorem \ref{thm:rsw-lemma-24} for some variables $a,b$. This is a symmetric polynomial in $a^{n-1},a^{n-2}b,\dots,ab^{n-2},b^{n-1}$ and, thus also in $a,b$.
This means the character specialization can be expressed as a polynomial in $a+b,ab$. If an action $I_{2}(n)\acts [n]$ is faithful, then $I_2(n)$ may be considered as a subgroup of permutation matrices of $\operatorname{GL}_{n}(\mathbb{C})$.

\begin{prop}\label{prop-generalization}
Let $n$ be odd, let $X$ be a finite set with $|X|=n$, and let $V$ be a $\operatorname{GL}_{n}(\mathbb{C})$-representation. Considering $I_2(n)$ as a subgroup of $\operatorname{GL}_n(\mathbb{C})$ through a faithful action $I_2(n)\acts X$, assume that $V$ has a basis indexed by $X$ which is permuted by $I_{2}(n)$ via $g(v_x)=v_{g(x)}$ for all $g\in I_2(n)$ and $x\in X$. Let $p$ be the unique polynomial in two variables such that
$$p(a + b, -ab) = \chi_{V}(\operatorname{diag}(a^{n - 1}, a^{n - 2}b, \dots, ab^{n - 2}, b^{n - 1}))$$
noting that the right-hand side is a symmetric function in $a$ and $b$. Then $(X, p, I_{2}(n))$ exhibits dihedral sieving.
\end{prop}
\begin{proof}
Let $C$ be a conjugacy class in $I_{2}(n)$ and $X = \frac{s + \sqrt{s^{2} + 4t}}{2}$, $Y = \frac{s - \sqrt{s^{2} + 4t}}{2}$ where $s = \chi_{z_{1}}(C)$ and $t = \chi_{-\det}(C)$. It is straightforward to check that the eigenvalues of any element in $C$ are $X^{n - 1}, X^{n - 2}Y, \dots, Y^{n - 1}$, and that $X + Y = \chi_{z_{1}}(C)$ and $XY = -\chi_{-\det}(C)$. Thus
\begin{align*}
\chi_{V}(C) &= \chi_{V}(X^{n - 1}, X^{n - 2}Y, \dots, XY^{n - 2}, Y^{n - 1}) \\
&= p(X + Y, -XY) \\
&= p(\chi_{z_{1}}(C), \chi_{-\det}(C)).
\end{align*}
and $V = p(z_{1}, -\det)$ in $\Rep(I_{2}(n))$.
\end{proof}

\begin{prop}\label{prop-uptoscalars}
Suppose $V$ is an $I_{2}(n)$-representation and $X$ is a finite $I_{2}(n)$-set indexing a basis $\{v_{x} : x \in X\}$ of $V$ which is permuted up to scalars, that is there is some one-dimensional representation $\rho : G \to \mathbb{C}^{\times}$ such that $g(v_{x}) = \rho(g)^{k}v_{g(x)}$ for all $g \in I_{2}(n)$ and $x \in X$. Suppose further that $V \cong p_{1}(z_{1}, -\det)$ and that $p_{1} = p_{2}p$ with $\rho \cong p_{2}(z_{1}, -\det)$. Then $(X, p, I_{2}(n))$ exhibits dihedral sieving.
\end{prop}
\begin{proof}
We have $\rho \mathbb{C}[X] = V = p_{2}(z_{1}, -\det)p(z_{1}, -\det)$ in $\Rep(I_{2}(n))$. Since $\rho$ is invertible in $\Rep(I_{2}(n))$ and $\rho = p_{2}(z_{1}, -\det)$, we can cancel it from both sides to get $\mathbb{C}[X] = p(z_{1}, -\det)$.
\end{proof}

\begin{cor}\label{cor-uptoscalars}
Let $n, X, V$ and $p$ be as in Proposition \ref{prop-generalization}, and suppose $p(s, t) = (-t^{k})u(s, t)$ and instead that the basis of $V$ is permuted up to scalars, that is $g(v_{x}) = (\det(g))^{k}v_{g(x)}$. Then $(X, u(s,t), I_{2}(n))$ exhibits dihedral sieving.
\end{cor}
\begin{proof}
Proposition \ref{prop-generalization} shows that $V = p(z_{1}, -\det)$ in $\Rep(I_{2}(n))$. Proposition \ref{prop-uptoscalars} then implies that $(X, u, I_{2}(n))$ exhibits dihedral sieving since $\det^{k} = (-t)^{k}|_{t=-\det}$.
\end{proof}

\begin{rem}
Instead of writing a symmetric polynomial in $a$ and $b$ as a polynomial in $a + b$ and $ab$, we can equivalently make the substitution $a = \frac{s + \sqrt{s^{2} + 4t}}{2}$ and $b = \frac{s - \sqrt{s^{2} + 4t}}{2}$, as in Proposition \ref{prop:fibonomial-thm}, which satisfies $a + b = s$ and $-ab = t$. In particular, we have $\{n\}_{s=a+b,t=-ab} = a^{n - 1} + a^{n - 2}b + \cdots + ab^{n - 2} + b^{n - 1}$. Stating Proposition \ref{prop-generalization} as we have done makes it clear that the resulting expression is always a polynomial.
\end{rem}

\begin{rem}
When $n$ is even, the eigenvalues of an element of $I_{2}(n)$ (again viewed as a subgroup of $GL_{n}(\mathbb{C})$) have a more complicated description. In particular, a reflection with two fixed points has $n/2 - 1$ eigenvalues which are $-1$ and $n/2 + 1$ eigenvalues which are $+1$, so they cannot be put into a geometric sequence of the form $X^{n - 1}, X^{n - 1}Y, \dots, X^{n - 1}$ as done in the proof of Proposition \ref{prop-generalization}. However, all reflections with no fixed points and rotations still have eigenvalues which can be written in this form.
\end{rem}

\begin{prop}
Let $n$ be odd and $X = \multiset{[n]}{k}$. Then the triple
\[\left(X, \stbinom{n + k  - 1}{k}_{s,t},I_{2}(n)\right)\]
exhibits dihedral sieving.
\end{prop}
\begin{proof}
Let $V=\operatorname{Sym}^k(\mathbb{C}^n)=V_{\lambda}$, the irreducible representation of $\operatorname{GL}_n(\mathbb{C})$ of highest weight $\lambda=(k)\vdash k$. Let it be equipped with the usual basis of symmetric tensors indexed by $k$-multisubsets of $[n]$. The action of $I_{2}(n)$ on this basis and on $k$-multisubsets of $[n]$ are the same, so we can apply Proposition \ref{prop-generalization}.
\end{proof}

The following lemma is useful for showing that dihedral sieving polynomials obtained from $GL_{n}(\mathbb{C})$ representations have certain product formulas, when the product formula is expressed as a rational function not known to be a polynomial.

\begin{lem}
Let $k$ be a field and suppose $f \in k(s, t)$ is a rational function such that $f(a + b, -ab)$ is a polynomial. Then $f$ is a polynomial.
\end{lem}
\begin{proof}
It is clear that $f(a + b, -ab)$ is symmetric in $a$ and $b$, and since it is a polynomial it is in the polynomial subring $k[a + b, -ab]$. Hence $f$ is a polynomial.
\end{proof}

The main example of the use of this lemma is to prove that the following analogue of the hook-content formula is a polynomial.

\begin{prop}\label{prop-hookcontent}
Let $(\lambda_{1}, \dots, \lambda_{n}) = \lambda \vdash k$ be a partition of $k$. Then
\[
s_{\lambda}(a^{n - 1}, a^{n - 2}b, \dots, ab^{n - 2}, b^{n - 1}) = \left((-t)^{b(\lambda)} \prod_{x \in \lambda} \frac{\{n + c(x)\}_{s,t}}{\{h(x)\}_{s,t}} \right)_{s = a + b, t = -ab}
\]
where $x \in \lambda$ runs over cells of $\lambda$, $h(x)$ is the hooklength of $x$ (total number of cells weakly to the right of $x$ or strictly below $x$), $c(x)$ is the hook content $j - i$ when cell $x$ is in row $i$ and column $j$, and
\[
b(\lambda) = \sum _{i} (i - 1)\lambda_{i}.
\]
\end{prop}
\begin{proof}
We have, using the $q$-hook content formula stated in Equation~\ref{eq-qhook},
\begin{align*}
s_{\lambda}(a^{n - 1}, a^{n - 2}b, \dots, ab^{n - 2}, b^{n - 1}) &= b^{(\deg s_{\lambda})(n - 1)}s_{\lambda}(1, a/b, (a/b)^{2}, \dots, (a/b)^{n - 1}) \\
&= a^{b(\lambda)} b^{-b(\lambda) + |\lambda|(n - 1)} \prod _{x \in \lambda} \frac{[n + c(x)]_{q=a/b}}{[h(x)]_{q=a/b}} \\
&= a^{b(\lambda)} b^{-b(\lambda) + |\lambda|(n - 1)} \prod _{x \in \lambda} \frac{b^{-(n + c(x) - 1)}\{n + c(x)\}_{s=a+b, t=-ab}}{b^{-(h(x) - 1)}\{h(x)\}_{s=a+b, t=-ab}} \\
&= a^{b(\lambda)} b^{-b(\lambda) + |\lambda|(n - 1)} \prod _{x \in \lambda} \frac{b^{-n - c(x))}\{n + c(x)\}_{s=a+b, t=-ab}}{b^{-h(x)}\{h(x)\}_{s=a+b, t=-ab}} \\
&= a^{b(\lambda)} b^{-b(\lambda) + |\lambda|(n - 1) + \sum _{x \in \lambda } h(x) - c(x) - n} \prod _{x \in \lambda} \frac{\{n + c(x)\}_{s=a+b, t=ab}}{\{h(x)\}_{s=a+b, t=-ab}}.
\end{align*}
where we use the fact that $[n]_{q=a/b}=b^{1-n}\{n\}_{s=a+b,t=-ab}$ via Proposition~\ref{prop:fibonomial-thm} to establish the third equality. For the exponent, we have
\begin{align*}
-b(\lambda) + |\lambda|(n - 1) + \sum _{x \in \lambda } h(x) - c(x) - n &= -b(\lambda) - |\lambda| + \sum _{x \in \lambda} h(x) - \sum _{x \in \lambda} c(x) \\
&= -b(\lambda) - |\lambda| + b(\lambda') + b(\lambda) + |\lambda| - b(\lambda') + b(\lambda) \\
&= b(\lambda)
\end{align*}
following from the identities
\begin{align*}
\sum _{x \in \lambda} c(x) &= b(\lambda') - b(\lambda) \\
\sum _{x \in \lambda} h(x) &= b(\lambda') + b(\lambda) + k
\end{align*}
where $\lambda'$ is the conjugate of $\lambda$.
Since $(ab)^{b(\lambda)} = (-t)^{b(\lambda)}|_{t=ab}$, the result follows.
\end{proof}

The technique of applying Theorem \ref{thm:rsw-lemma-24} to prove cyclic sieving for $k$-subsets and $k$-multisubsets of $[n]$ immediately generalizes with the above proposition.

\begin{prop}\label{prop:ksubsets}
Let $n$ be odd and let $X = \binom{[n]}{k}$. Then the triple
\[\left(X, \stbinom{n}{k}_{s,t}, I_{2}(n)\right)\]
exhibits dihedral sieving.
\end{prop}
\begin{proof}
Using Proposition \ref{prop-hookcontent} and Corollary \ref{cor-uptoscalars}, it suffices to show that there is a basis
\[\left\{w_{A} : A \in \binom{[n]}{k}\right\}\] of $\Lambda^{k}\mathbb{C}^{n}$ which is permuted up to scalars by the action of $I_{2}(n)$, that is $x(w_{i}) = \det(x)^{\binom{k}{2}}w_{x(i)}$ for all $x \in I_{2}(n)$.

We construct the basis as follows. First, group $\binom{[n]}{k}$ into orbits under the action by the cyclic group $C_{n}$. For each orbit we will choose a distinguished representative such that if $A$ is distinguished then $r(A)$ is also distinguished. Given a subset $A = \{i_{1}, \dots, i_{k}\}$ with $i_{1} < \cdots < i_{k}$, let $v_{A} = e_{i_{1}} \wedge \cdots \wedge e_{i_{k}}$.

Since the reflection in $I_{2}(n)$ has order two, each orbit of $\binom{[n]}{k}$ under the $I_{2}(n)$ action is a union of either two or one orbits of the $C_{n}$ action. Suppose we have two $C_{n}$-orbits whose union is a single $I_{2}(n)$-orbit. Then choose an arbitrary element $A$ of one orbit to be distinguished, and take $B = s(A)$ to be the distinguished element of the other orbit. Form the set
\[
\{v_{A}, c(v_{A}), \dots, c^{j - 1}(v_{A}), (-1)^{\binom{k}{2}}v_{B}, (-1)^{\binom{k}{2}}c(v_{B}), \dots, (-1)^{\binom{k}{2}}c^{j-1}(v_{B})\} \subseteq \wedge^{k}\mathbb{C}^{n}
\]
where the orbits each contain $j$ subsets. These vectors are linearly independent and permuted up to scalars by all elements of $\{r^{0}, \dots, r^{n - 1}, s\}$ since this holds for $v_{A}$ and because of the relation $sr^{i} = r^{-i}s$. These elements also generate $I_{2}(n)$, so in fact all elements of $I_{2}(n)$ permute these vectors up to the desired scalars.

In the case where we have a $C_{n}$-orbit which is also an $I_{2}(n)$-orbit, for a given $A$ in the orbit we have $s(A) = r^{j}(A)$ for some $j$, so $r^{-j}(s(A)) = A$. Now take the subset $\{v_{A}, c(v_{A}), \dots, c^{n-1}(v_{A})\}$. Then
\[ r^{-j}(s(v_{A})) = e_{i_{k}} \wedge \cdots \wedge e_{i_{1}} = (-1)^{\binom{k}{2}}(e_{i_{1}} \wedge \cdots \wedge e_{i_{k}}) = (-1)^{\binom{k}{2}}v_{A}. \]
We have the relation $(r^{-j}s)r^{i} = r^{-i}(r^{-j}s)$ analogous to the one above, so this set is also permuted up to scalars.
Taking the union of these subsets gives the desired basis of $\Lambda^{k}\mathbb{C}^{n}$.
\end{proof}

\begin{rem}
Dihedral sieving for $k$-subsets and $k$-multisubsets of $[n]$ can be also proven using direct enumeration.
\end{rem}


\section{\texorpdfstring{Dihedral Action on Non-Crossing Partitions of $\{1,\ldots,n\}$}{Dihedral Action on Non-Crossing Partitions of \{1, ..., n\}}}

There is an action by $I_2(n)$ on the non-crossing partitions of $[n]$ and also on the non-crossing partitions with a fixed number of parts. The character values for the corresponding representations of both actions were studied by Ding in 2016 in \cite{ding}. We show, for odd $n$, these actions are both instances of dihedral sieving using the natural $(s,t)$-analogue of the Catalan number as the generating polynomial.

\begin{prop}
Let $n$ be odd and $X = \{\text{non-crossing partitions of $[n]$}\}$. Then the triple
\[\left(X, \frac{1}{\{n + 1\}_{s,t}}\stbinom{2n}{n}_{s,t}, I_{2}(n)\right)\]
exhibits dihedral sieving.
\end{prop}
\begin{proof}
Again, let $\xi_n$ be the primitive $n$-th root of unity $e^{2\pi i/n}$ and let $C$ be a conjugacy class of $I_2(n)$.

First, we compute the character values of $\mathbb{C}[X]$. We have that $C_n(\xi_n^{\ell})$ counts the fixed points of the action of $\{r^{\ell},r^{n-\ell}\}$ by cyclic sieving. \cmt{}  Next, we have that $\binom{n}{\lfloor n/2\rfloor}$, or $C_n(-1)$, is the number of fixed points of $\{r,sr,sr^2,\ldots\}$ by Theorem 2.1.5 of \cite{ding}.

Using the same notation as before, we consider the generalized  Catalan number/sequence in Section 5 of \cite{acms} ${C_{\{n\}}}_{s,t}:=\frac{1}{\{n+1\}_{s,t}}\stbinom{n}{k}_{s,t}$ where we use the specialization $(s,t)=(z_1,-\det)$. Then, we use the fact that
$$\{n+1\}_{s,t}=Y^{n}[n+1]_{X/Y}=\begin{cases}
[n+1]_{\xi_n^{\ell}} & \text{if $C=\{r^{\ell},r^{n-\ell}\}$}\\
[n+1]_{\xi_2} & \text{if $C=\{s,sr,sr^2,\ldots\}$}
\end{cases}$$
to get
$${C_{\{n\}}}_{(z_1(C),-\det(C))}=\begin{cases}
C_n(q)|_{q=\xi_n^{\ell}} & \text{if $C=\{r^{\ell},r^{n-\ell}\}$}\\
C_n(q)|_{q=\xi_2} & \text{if $C=\{s,sr,sr^2,\ldots\}$}
\end{cases}$$
where $C_{\{n\}}$ is well-known to be a polynomial in $s,t$ with integral coefficients. The claim follows.
\end{proof}

\begin{prop}
Let $n$ be odd and $X = \{\text{non-crossing partitions of $[n]$ with $n - k$ blocks}\}$. Then the triple
\[\left(X, \frac{1}{\{n\}_{s,t}}\stbinom{n}{k}_{s,t}\stbinom{n}{k + 1}_{s,t}, I_{2}(n)\right)\]
exhibits dihedral sieving.
\end{prop}
\begin{proof}
First, we claim the character values of $\mathbb{C}[X]$ are $N(n,k;\xi_n^{\ell})$ for the conjugacy class $\{r^{\ell},r^{n-\ell}\}$ and $N(n,k;-1)$ for the conjugacy class $\{s,sr,sr^2,\ldots\}$, where $N(n,k)=|X|$ is the \emph{Narayana number} and 
\begin{equation}
N(n,k;q):=\frac{1}{[n]_q}\qbinom{n}{k}_q\qbinom{n}{k+1}_q q^{k(k+1)}
\end{equation}
the $q$-analogue of the Narayana number. In fact, both of these are already known. The character values of rotation conjugacy classes follow from Theorem 7.2 of \cite{rsw}. For the case of reflections, the character values are given by $N(n,k;-1)$ by Theorem 3.2.7 of \cite{ding}.

Now, consider the $(s,t)$-analogue of the Narayana number. This is a polynomial in $s,t$ for the same reasons $N(n,k;q)$ is a polynomial in $q$:
\begin{align*}
\frac{1}{\{n\}_{s,t}}\stbinom{n}{k}_{s,t}\stbinom{n}{k+1}_{s,t}&=\frac{1}{\{n-k\}_{s,t}}\stbinom{n}{k+1}_{s,t}\stbinom{n-1}{k}_{s,t}\\
&=\frac{1}{\{k+1\}_{s,t}}\stbinom{n}{k}_{s,t}\stbinom{n-1}{k}_{s,t}\\
&=\stbinom{n-1}{k}_{s,t}\stbinom{n+1}{k+1}_{s,t}-\stbinom{n}{k}_{s,t}\stbinom{n}{k+1}_{s,t}
\end{align*}
Next, the $(s,t)$-analogue of the Narayana number becomes, by Proposition \ref{prop:fibonomial-thm},
\[
	\frac{1}{\{n\}_{s,t}}\stbinom{n}{k}_{s,t}\stbinom{n}{k+1}_{s,t}=Y^{(1-n)+k(n-k)+(k+1)(n-k-1)}N(n,k;X/Y)
\]
where the exponent of $Y$ simplifies to $2k(n-k)-2k$ and if $Y=\xi_n^{-\ell}$, then $n|k\gcd(n,\ell)\implies n|\ell k$ so that the power of $Y$ goes to $1$ when the $q$-binomial $\qbinom{n}{k}_{q=X/Y}$ is non-zero. Thus, the claim follows.
\end{proof}

\section{\texorpdfstring{Dihedral Action on Triangulations of an $n$-gon}{Dihedral Action on Triangulations of an $n$-gon}}

The dihedral action on the triangulations of a regular $n$-gon involves Garsia and Haiman's $(q,t)$-Catalan number. The proof is by direct enumeration.

\begin{thm}
Let $n$ be odd and let $X$ be the set of triangulations of a regular $n$-gon. Then, the triple
\[
	\left(X,(-r)^{\binom{n-2}{2}}p_n(s,-r),I_2(n)\right)
\]
exhibits dihedral sieving, where $p_n(s,-r)=C_{n-2}(q,t)$ is the $q,t$-Catalan number of Garsia and Haiman written using the basis $s=q+t,-r=qt$, by virtue of it being a symmetric polynomial.
\end{thm}

\begin{proof}
We verify this directly by computation. First, we consider the conjugacy class $C$ of rotations. Letting $q=\xi_n^{\ell}$ where $\xi_n$ is a primitive $n$-th root of unity for $\ell\in[n]$ and $t=q^{-1}$, we have $z_1(C)=s=q+q^{-1}$ and $r=qq^{-1}=1=\det(C)$ so that the claim follows from the identity
\[
	r^{\binom{n-2}{2}}p_n(s,r)=q^{\binom{n-2}{2}}C_{n-2}(q,q^{-1})=C_{n-2}(q)
\]
where $C_{n-2}(q)$ is MacMahon's $q$-Catalan number and where we use the fact that $q^{\binom{n-2}{2}}=1$ when $n$ is a multiple of $3$ which is the only instance where $C_{n-2}(q)$ is non-zero since a triangulated polygon with an odd number of sides is invariant under a rotation only if the triangle containing its incenter is.

For the conjugacy class $C$ of reflections, letting $q=-1$ and $t=1$, we have $z_1(C)=s=q+t=0$ and $r=-1=\det(C)$ so that
\[
	r^{\binom{n-2}{2}}p_n(s,r)=(-1)^{\binom{n-2}{2}}C_{n-2}(-1,1)
\]
Let $\tilde{C}_n(q)$ be the Carlitz-Riordan $q$-Catalan number. It is well known that $\tilde{C}_{n-2}(-1)=C_{n-2}(-1,1)$ and that
\[
	\tilde{C}_{n+1}(-1)=\sum_{k=0}^n (-1)^k\tilde{C}_k(-1)\tilde{C}_{n-k}(-1)
\]
with $\tilde{C}_0(-1)=1$. First, we claim $\tilde{C}_n(-1)=0$ for $n>0$ even. This is true by induction and evaluating the above formula. Thus, we have for $k\in\mathbb{N}$,
\[
	\tilde{C}_{2k-1}(-1)=-\sum_{i=1,2\nmid i}^{2k-2} \tilde{C}_i(-1)\tilde{C}_{2k-2-i}(-1)
\]
It suffices to show $(-1)^{\binom{2k-1}{2}}\tilde{C}_{2k-1}(-1)=C_{k-1}$, the ordinary $(k-1)$-th Catalan number. Noting the indices $i$ and $2k-2-i$ are both odd in each summand above, it is not hard to see that
\[
	\binom{i}{2}+\binom{2k-2-i}{2}+1\equiv \binom{2k-1}{2}\pmod{2}
\]
meaning, by induction, that the sum from before can be reparametrized as
\[
	(-1)^{\binom{2k-1}{2}}\tilde{C}_{2k-1}(-1)=\sum_{i=0}^{k-2} C_iC_{k-2-i}=C_{k-1}
\]
\end{proof}
Generalizing the ordinary Catalan numbers, the little Schr\"{o}der numbers $s_{n-2}$ count the number of ways to dissect a convex $n$-gon with any number of noncrossing diagonals. It has a decomposition into a sum of Kirkman numbers
\[
	\frac{1}{n+k}\binom{n+k}{k+1}\binom{n-3}{k}
\]
each of which counts the number of ways to dissect a convex $n$-gon using $k$ noncrossing diagonals. The $q$-analogue of the Kirkman number gives rise to a cyclic sieving phenomenon on the set of such dissections, as seen in Example~\ref{ex:sieving-examples} (6). The little Schr\"{o}der numbers and the Kirkman numbers have bivariate analogues, discussed in Chapter 4 of \cite{haglund}, which specialize to Garsia and Haiman's $(q,t)$-Catalan number.

Let $S_{n,d}(q,t)$ be the $q,t$-Schr\"{o}der number defined on pp. 60 of \cite{haglund}. Let $\tilde{S}_{n,d}(q,t)$ be the ``little'' $q,t$-Schr\"{o}der number defined recursively by
\begin{align*}
	\tilde{S}_{n,0}(q,t)&=S_{n,0}(q,t)\\
	S_{n,d}(q,t)&=\tilde{S}_{n,d}(q,t)+\tilde{S}_{n,d-1}(q,t)
\end{align*}
for $0\leq d\leq n$ according to pp. 61 of \cite{haglund}. By Corollary  4.8.1 of \cite{haglund}, we can write
\[
	S_{n,d}(q,1/q)=\frac{q^{\binom{d}{2}-\binom{n}{2}}}{[n-d+1]_q}\binom{2n-d}{n-d,n-d,d}_q
\]
where $S_{n,d}(q,1/q)$ is symmetric in $q$ and $1/q$ and, thus, so is $\tilde{S}_{n,d}(q,1/q)$. However, it is not known if $S_{n,d}(q,t)=S_{n,d}(t,q)$ in general. It is also true that $\tilde{S}_{n,0}(q,t)=C_n(q,t)$. This suggests a dihedral sieving phenomenon can arise using this polynomial. Using the above formula, we can evaluate $\tilde{S}_{n,d}(q,1/q)$ at specific points and compare them to the $q$-Kirkman numbers. Computational evidence seems to suggest the following is true.

\begin{conj}
The little $q,t$-Schr\"{o}der number is symmetric in $q$ and $t$. Furthermore, letting $n$ be odd and letting $X$ be the set of dissections of a regular $(n+2)$-gon using $k$ noncrossing diagonals, the triple
\[
	\left(X,(-r)^{f(n,k)}p_n(s,-r),I_2(n)\right)
\]
exhibits dihedral sieving, where $p_n(s,r)=\tilde{S}_{n,n-1-k}(q,t)$ is the little $q,t$-Schr\"{o}der number written using the basis $s=q+t,-r=qt$ and $f(n,k)$ is a function of $n$ and $k$.
\end{conj}

\section{\texorpdfstring{Further Questions}{Further Questions}}
\subsection{\texorpdfstring{The case of even $n$}{The case of even n}}
Each of our proofs of instances of $I_2(n)$-sieving specifically relied on $n$ being odd. In the case of $n$ even, empirical evidence seems to suggest any possible instances of dihedral sieving for $k$-subsets, $k$-multisubsets, or noncrossing partitions will not be given by taking an obvious $(s,t)$-analogue polynomial when specializing to generators of the representation ring of $I_2(n)$. This is especially true since the representation ring is different for $n$ odd and even. Nonetheless, the character values for each of these group actions can be shown, in a similar manner to the odd case, to be $\mathbb{Z}$-linear combinations of $q$-analogues ($q$-binomials for the $k$-subsets and $k$-multisubsets actions and $q$-Catalan numbers for the noncrossing partitions action) evaluated at certain values, as shown in Table \ref{table-character}.

We expect that exhibiting dihedral sieving for even $n$ is possible but more difficult, especially for $k$-subsets and $k$-multisubsets using product formulas, and offer some evidence as to why. Consider the identity
\[
\operatorname{Sym}^{k}(X \oplus Y) \cong \bigoplus _{i = 0} ^{k} \operatorname{Sym}^{i}(X) \otimes \operatorname{Sym}^{k - i}(Y).
\]
Noting that $\Rep(SO(2)) \cong \mathbb{Z}[q, q^{-1}]$ and $\operatorname{Sym}^{k}(q^{m}) = q^{km}$, an inductive argument shows that $\operatorname{Sym}^{k}(1 + q + \cdots + q^{n - 1}) = \qbinom{n + k - 1}{k}_{q}$ in $\Rep(SO(2))$. The identity suggests that we should consider $\Rep(O(2))$ to study dihedral sieving, but this approach fails since the representation ring of $O(2)$ behaves more similarly to the case when $n$ is odd. In particular, all reflections in $O(2)$ are conjugate and the only one-dimensional irreducible representations of $O(2)$ are the trivial and determinant representations. Moreover, Fibonomial coefficients do not seem to describe symmetric powers of $O(2)$ representations in the same way as $q$-analogues do for $SO(2)$ representations, although there may be some alternative generalization of $q$-binomial coefficients which do.

Interestingly, the same discrepancy between even and odd $n$ occurs when we consider real representation rings of cyclic groups, i.e.\ the proper subring of $\Rep(C_{n})$ generated by all representations of the form $\mathbb{C} \otimes V$ where $V$ is a representation of $C_{n}$ over $\mathbb{R}$. The real irreducible representations of $C_{n}$ are all restrictions of irreducible representations of $I_{2}(n)$, and in particular there is an additional one-dimensional irreducible real representation when $n$ is even, while for $n$ odd and $SO(2)$ the only one-dimensional irreducible real representation is the trivial one. Thus it may useful to exhibit cases of cyclic sieving using generators of the real representation ring (regarded as a proper subring of the complex representation ring) before attempting dihedral sieving for even $n$. Note that since the irreducible real representations of $C_{n}$ are restrictions of $I_{2}(n)$ irreducibles, dihedral sieving for even $n$ would directly exhibit cases of cyclic sieving using generators of the real representation ring.

\section{Appendix}

\FloatBarrier
\begin{table}[H]\label{table-character}
\begin{tabular}{|c|c|c|c|}
\hline
$n$ odd & $e$ & $r^{\ell},r^{n-\ell}$ & $sr,sr^2,sr^3,\ldots$\\
 \hline
$\mathbbm{1}$ & $1$ & $1$ & $1$ \\
\hline
$\det$ & $1$ & $1$ & $-1$\\
\hline
$\chi_m$ & $2$ & $2\cos\left(\frac{2\pi m\ell}{n}\right)$ & $0$\\
\hline
$\chi_{\text{$k$-subsets}}$ & $\binom{n}{k}$ & $\qbinom{n}{k}_{q=\xi_n^{\ell}}$ & $\qbinom{n}{k}_{q=\xi_2}$\\
\hline
$\chi_{\NC(n)}$ & $C_n$ & $C_n(q)|_{q=\xi_n^{\ell}}$ & $C_n(q)|_{q=\xi_2}$\\
\hline
$\chi_{\text{triangulations}}$ & $C_{n-2}$ & $C_{n-2}(q)|_{q=\xi_n^{\ell}}$ & $\frac{2}{n-1}C_{n-3}(q)|_{q=\xi_2}$\\
\hline
$\chi_{NC(n,k)}$ & $N(n,k)$ & $N(n,k;\xi_n^{\ell})$ & $N(n,k;\xi_2)$\\
\hline
$\chi_{\text{$k$-multisubsets}}$ & $\binom{n+k-1}{k}$ & $\qbinom{n+k-1}{k}_{q=\xi_n^{\ell}}$ & $\qbinom{n+k-1}{k}_{q=\xi_2}$\\
\hline
\end{tabular}

\vspace{1em}

\begin{tabular}{|c|c|c|c|c|}
\hline
$n$ even & $e$ & $r^{\ell},r^{n-\ell}$ & $sr,sr^3,sr^5,\ldots$ & $s,sr^2,sr^4,\ldots$\\
 \hline
$\mathbbm{1}$ & $1$ & $1$ & $1$ & $1$ \\
\hline
$\det$ & $1$ & $1$ & $-1$ & $-1$\\
\hline
$\chi_b:=\begin{cases} \langle r^2,s\rangle\mapsto 1\\ r\mapsto -1\end{cases}$ & $1$ & $(-1)^{\ell}$ & $-1$ & $1$\\
\hline
$\chi_b\cdot \det$ & $1$ & $(-1)^{\ell}$ & $1$ & $-1$\\
\hline
$\chi_m$ & $2$ & $2\cos\left(\frac{2\pi m\ell}{n}\right)$ & $0$ & $0$\\
\hline
$\chi_{\text{$k$-subsets}}$ & $\binom{n}{k}$ & $\qbinom{n}{k}_{q=\xi_n^{\ell}}$ & $\qbinom{n}{k}_{q=\xi_2}$ & $\qbinom{n}{k}_{q=\xi_2}+2\qbinom{n-2}{k-1}_{q=\xi_2}+\qbinom{n-2}{k-2}_{q=\xi_2}$\\
\hline
$\chi_{\NC(n)}$ & $C_n$ & $C_n(q)|_{q=\xi_n^{\ell}}$ & $C_n(q)|_{q=\xi_2}$ & $C_n(q)|_{q=\xi_2}$\\
\hline
$\chi_{\text{triangulations}}$ & $C_{n-2}$ & $C_{n-2}(q)|_{q=\xi_n^{\ell}}$ & $0$ & $\frac{4}{n}C_{n-2}(q)|_{q=\xi_2}$\\
\hline
$\chi_{\NC(n,k)}$ & $N(n,k)$ & $N(n,k,\xi_n^{\ell})$ & $N(n,k,\xi_2)$ & $N(n,k,\xi_2)$\\
\hline
$\chi_{\text{$k$-multisubsets}}$ & $\binom{n+k-1}{k}$ & $\qbinom{n+k-1}{k}_{q=\xi_n^{\ell}}$ & $\qbinom{n+k-1}{k}_{q=\xi_2}$ & $\qbinom{n+k-2}{k}_{q=\xi_2}+2\qbinom{n+k-3}{k-1}_{q=\xi_2}+\qbinom{n-3+k}{k-2}_{q=\xi_2}$\\
\hline
\end{tabular}
\caption{Character values for various $I_{2}(n)$ actions.}
\end{table}

\section*{Acknowledgments} 
This research was carried out as part of the 2017 summer REU program at the School of Mathematics, University of Minnesota, Twin Cities, and was supported by NSF RTG grant DMS-1148634. The authors would like to thank Victor Reiner, Pavlo Pylyavskyy, and Benjamin Strasser for their mentorship and support.

\end{document}